\documentclass[A4,11pt]{article}
\usepackage{geometry}

\usepackage{authblk}

\usepackage{amssymb,amsmath,amsthm,mathtools}

\usepackage{enumitem}
\usepackage[linktoc=all,hidelinks]{hyperref}

\DeclareMathOperator{\diag}{diag}
\DeclareMathOperator{\re}{Re}
\DeclareMathOperator{\sgn}{sgn}
\DeclareMathOperator*{\esssup}{ess\,sup}

\newcommand{\bra}[1]{\langle#1\rangle}
\newcommand{\bras}[1]{\left[#1\right]}

\newcommand{\latt}[2]{#1\mathbb Z\times #2\mathbb Z}
\newcommand{\gframe}[2]{\mathcal G(#1,#2)}
\newcommand{\glframe}[3]{\mathcal G(#1,\latt{#2}{#3})}

\newtheorem*{Thm*}{Theorem}
\newtheorem{Thm}{Theorem}
\newtheorem{Lm}{Lemma}
\newtheorem{Prop}{Proposition}

\title{Frame Sets and Zeros of Zak Transforms \\ of Extended Gaussians}
\author[1]{Wenchang Sun\thanks{\tiny sunwch@nankai.edu.cn; This work was supported by the National Natural Science Foundation of China (12571104)}}
\affil[1]{School of Mathematical Sciences and LPMC, Nankai University, Tianjin, China}
\author[2]{Weiqi Zhou\thanks{\tiny zwq@xzit.edu.cn}}
\affil[2]{School of Mathematics and Statistics, Xuzhou University of Technology, Xuzhou, Jiangsu Province, China}
\date{}							
						
\begin{document}
\maketitle
\begin{abstract}
Let $a,b,c\in\mathbb C$ with $\re(a)<0$, we show that the extended Gaussian $e^{ax^2+bx+c}$ has maximal frame set (i.e., its frame set consists of precisely all positive pairs $(\alpha,\beta)$ with $\alpha\beta<1$), and its Zak transform has a unique simple zero in the unit square $[0,1)^2$ (in particular, the zero is at the center of the unit square if $b=0$). These statements extend the same results of the usual Gaussian (the cases when $a<0$ and $b,c\in\mathbb R$), and add more instances to the observation that if a continuous Wiener function has maximal frame set, then its Zak transform has a unique simple zero in the unit square. The proof of the maximality of the frame set combines metaplectic representation with a classical density result of the standard Gaussian. The proof of the uniqueness of the zero relies on properties of the theta function.  \\

{\noindent
}
\end{abstract}

\section{Introduction}
Let
\[T_t:f(x)\mapsto f(x-t), \quad M_{\xi}: f(x)\mapsto e^{2\pi i\xi x}f(x)\]
be respectively the \emph{translation (i.e., time shift) operator} and the \emph{modulation (i.e., frequency shift) operator} on $L^2(\mathbb R)$. Given $g\in L^2(\mathbb R)$ and $\Lambda\subset\mathbb R^2$, a \emph{Gabor system} on $L^2(\mathbb R)$ with \emph{window} $g$ and \emph{support} $\Lambda$ is defined as and denoted by
\[\gframe{g}{\Lambda}=\{M_aT_bg: (b,a)\in\Lambda\}.\]
A typical choice of $\Lambda$ would be the lattice $\latt{\alpha}{\beta}$ with $\alpha,\beta>0$. And the corresponding Gabor system is
\[\glframe{g}{\alpha}{\beta}=\{M_{\beta j}T_{\alpha k}g: j,k\in\mathbb Z\}.\]
Gabor systems are fundamental building blocks in signal processing and time-frequency analysis since they offer a natural way to discretize a function on the time-frequency plane. A trivial example of a Gabor system is the orthonormal basis of $L^2(\mathbb R)$ formed by taking the window to be the characteristic function on the unit interval $\chi_{[0,1)}$ and $\alpha=\beta=1$.

A set of functions $\{h_k\}_{k\in\mathbb Z}$ is called a \emph{frame} in a Hilbert space $H$ if there exist constants $A,B>0$ so that
\[A\|f\|_{H}^2\le\sum_{k\in\mathbb Z}|\bra{f,h_k}_H|^2\le B\|f\|_H^2\]
holds for any $f\in H$. If $\{h_k\}_{k\in\mathbb Z}$ is a frame, then both the analysis operator $f\mapsto \{\bra{f,h_k}_H\}_{k\in\mathbb Z}$ and the synthesis operator $\{\bra{f,h_k}_H\}_{k\in\mathbb Z}\mapsto f$ are bounded from above and below. Hence it allows a stable decomposition and reconstruction of $f$ into and from the coefficients $\{\bra{f,h_k}_H\}_{k\in\mathbb Z}$.

A \emph{Gabor frame} on $L^2(\mathbb R)$ is a Gabor system that forms a frame on $L^2(\mathbb R)$. The set of parameters $(\alpha,\beta)$ that makes $\glframe{g}{\alpha}{\beta}$ a frame is called the \emph{frame set} of $g$ and denoted by $F(g)$. The problem of determining $F(g)$ for an arbitrary function $g$ is a long standing difficult open problem in this field. Even for simple window functions such as $\chi_{[0,1)}$, the corresponding frame set could be fairly complicated (see \cite{janssen2003a}). The classical density theory points out that if $\alpha\beta>1$, then $\glframe{g}{\alpha}{\beta}$ is incomplete and thus can not possibly be a frame (see e.g., \cite[Theorem 3a]{heil2007} or \cite{baggett1990}).

The critical case $\alpha\beta=1$ is already completely characterized by the Zak transform of the window function. More importantly it has been excluded from frame sets of functions with good time-frequency concentration. This is due to the fact that in such case $\glframe{g}{\alpha}{\beta}$ is a frame if and only if the Zak transform of $g$ is essentially bounded away from both $0$ and $\infty$ (see e.g., \cite[Corollary 8.3.2]{groechenig2001}), and it is well known that if the transformed function is continuous, then it always admits zeros (see e.g., \cite[Lemma 8.4.2]{groechenig2001}) and thus can not be lowered bounded. A class of functions that will guarantee continuity after the Zak transform is the set of continuous Wiener functions.
\emph{Wiener functions} on $\mathbb R$ form a
Banach space equipped with the norm (see e.g., \cite[Chapter 6.1]{groechenig2001})
\[\|g\|_{W(\mathbb R)}=\sum_{k\in\mathbb Z}\esssup_{x\in[0,1)}|g(x+k)|<\infty.\]
It is thus a mixed norm space that is locally of $L^{\infty}$ bound and globally of $\ell^1$ decay. The subspace of continuous functions in it is denoted by $W_0(\mathbb R)$. That the frame set of a $W_0(\mathbb R)$ function (or its Fourier transform) can not contain any pair $(\alpha,\beta)$ with $\alpha\beta=1$ is known as the Amalgam Balian-Low theorem.

We shall say that a function $g$ \emph{has maximal frame set} (or $F(g)$ is maximal), if $F(g)$ contains all pairs $(\alpha,\beta)$ with $\alpha\beta<1$ and $\alpha,\beta>0$. A primal question is then to identify all $W_0(\mathbb R)$ functions whose frame sets are maximal. It was first conjectured in \cite{daubechies1990} that Gaussians possess such a property. This is proved in \cite{lyubarski1992, seip1992, wallsten1992}. Later the hyperbolic secant $1/(e^x+e^{-x})$ \cite{janssen2002}, the one sided exponential $e^{-x}\chi_{\mathbb R^+}(x)$ \cite{janssen1996} (actually this is not a $W_0(\mathbb R)$ function since it is not even continuous, its Zak transform has no zero and its frame set contains also the critical case $\alpha\beta=1$, see \cite[Section 3.2.1]{janssen2003a}) and the two sided exponential $e^{-|x|}$ \cite{janssen2003} are also found to have this property. Then apparently automorphisms of the $L^2$ space such as translations, modulations, dilations, scalings and Fourier transforms of these functions will inherit this property.

It has been noticed that these functions are not only positive but also totally positive. A function $g$ is called \emph{totally positive} if for all $n\in\mathbb N$ and any two sequences of non-decreasing numbers $x_1\le\ldots\le x_n$ and $y_1\le\ldots\le y_n$, the determinant of the matrix $[g(x_i-y_j)]_{i,j=1}^n$ is non-negative. Alternative $g$ is totally positive if and only if its Fourier transform factors into
\begin{equation} \label{EqTP}
\hat g(\xi)=ce^{-\gamma\xi^2}e^{2\pi i\nu\xi}\prod_{j=1}^{N}(1+2\pi i\nu_j\xi)e^{-2\pi i\nu_j\xi},
\end{equation}
where $c>0$, $\gamma\ge0$, $\nu,\nu_j\in\mathbb R$, $N\in\mathbb N\cup\{\infty\}$ and $0<\gamma+\sum_{j=1}^N\nu_j^2<\infty$. This is characterized in \cite{schoenberg1947,schoenberg1951,schoenberg1953}. In a series of papers \cite{groechenig2023, groechenig2018, groechenig2013} it is proved that if $N$ is finite or if $N=\infty$ but $\alpha\beta\in\mathbb Q$, then $\alpha\beta<1$ is in the frame set of $g$ (in particular, if $g$ is at the same time in the Feichtinger algebra, then $F(g)$ is open \cite{feichtinger2003}, and thereof maximal), which covers and also greatly extends all aforementioned known cases. Mysteriously the Zak transforms of all these functions (except the one-sided exponential, for which there is no zero) possess a unique simple zero in the unit square (see \cite{baxter1996, kloos2014, kloos2015, vinogradov2017}).

Is total positivity necessary for maximality of frame sets? If we modulate a totally positive function, or simply multiply it by a negative number, then it is not totally positive any more but its frame set stays invariant. In fact it is also pointed out in \cite{belov2023} that certain linear combinations of Cauchy kernels do admit maximal frame sets, these functions are neither totally positive nor time-frequency shifts or Fourier transforms of totally positive functions, but they are not Wiener functions either. Therefore to discuss necessity it perhaps makes sense to phrase it as ``is total positivity necessary for real non-negative Wiener functions to have maximal frame sets?''

The purpose of this paper is to report more instances of $W_0(\mathbb R)$ functions whose frame sets are maximal and whose Zak transforms also have unique simple zeros in the unit square. Let $\boldsymbol a,\boldsymbol b,\boldsymbol c\in\mathbb C$ with $\re(\boldsymbol a)<0$, we call
\[\varphi(x;\boldsymbol a,\boldsymbol b,\boldsymbol c)=e^{\boldsymbol ax^2+\boldsymbol bx+\boldsymbol c}, \quad x\in\mathbb R\]
an \emph{extended Gaussian}, and refer to $\varphi(x)=\varphi(x;-\pi,0,0)=e^{-\pi x^2}$ as the \emph{standard Gaussian}. We prove that 

\begin{Thm*}
Let $\boldsymbol a,\boldsymbol b,\boldsymbol c\in\mathbb C$ with $\re(\boldsymbol a)<0$, then the extended Gaussian $\varphi(x;\boldsymbol a,\boldsymbol b,\boldsymbol c)=e^{\boldsymbol ax^2+\boldsymbol bx+\boldsymbol c}$ has maximal frame set, and its Zak transform has a unique simple zero in the unit square $[0,1)^2$. In particular, if $\boldsymbol b=0$, then the zero is at $(1/2,1/2)$.
\end{Thm*}

The proof of the maximality combines metaplectic representation with a classical density result of the standard Gaussian. The proof of the uniqueness of the zero relies on properties of the theta function. We also include an exposition of known facts about zeros of the Zak transform in the appendix (as they are not directly related to this paper, but seem to be playing some mysterious roles in the analysis of frame sets).

\section{Notations and conventions}
In this article the Fourier transform of a function $f$ is denoted by $\mathcal F(f)$ or $\hat f$, and for $f\in L^1(\mathbb R)$ it takes the following integral form:
\[\mathcal Ff(\xi)=\hat f(\xi)=\int_{\mathbb R}f(x)e^{-2\pi i\xi x}dx.\]
The inverse Fourier transform of $f$ is denoted by $\mathcal F^{-1}f$.

It is customary and convenient to use the following notation for time-frequency shifts:
\[\pi(z)=\pi(a,b)=M_aT_b, \quad z=(a,b)^T\in\mathbb R^2.\]
(Hereafter $T$ in the superscript stands for transposition). Let $z'=(a',b')$, we also remind the readers about the commutativity relation that
\begin{equation} \label{EqComm}
\pi(z)\pi(z')=M_aT_bM_{a'}T_{b'}=e^{2\pi i(ab'-a'b)}M_{a'}T_{b'}M_{a}T_{b}=e^{2\pi i\bras{z,z'}}\pi(z')\pi(z),
\end{equation}
where 
\[\bras{z,z'}=ab'-a'b\]
is the \emph{symplectic form}. Some classical results were also stated using the \emph{Schr\"odinger representation} (see e.g., \cite[Chapter 1.3]{folland1989} and \cite[Chapter 6.1]{gosson2006}):
\begin{equation} \label{EqDefRho}
\rho(z)=\rho(b,a)=e^{\pi iab}T_{b}M_{a}, \quad z=(b,a)^T\in\mathbb R^2.
\end{equation}
If we take $z=(a,0)$ and $z'=(0,b)$ in \eqref{EqComm}, then it is easy to derive that
\begin{equation} \label{EqPiRho}
\pi(a,b)=e^{\pi iab}\rho(b,a).
\end{equation}

The dilation of a function $g$ by a parameter $\gamma>0$ will be denote by
\[g_{\gamma}(x)=g(\gamma x).\]
It is easy to see that
\[\|g\|_{L^2(\mathbb R)}^2=\int_{\mathbb R}|g(x)|^2dx=\gamma\int_{\mathbb R}|g(\gamma y)|^2dy=\gamma\|g_{\gamma}\|_{L^2(\mathbb R)}^2.\]
Similarly if $D_{\gamma}=\diag(1/\gamma,\gamma)$, then for any $z=(a,b)^T\in\mathbb R^2$ we have
\begin{equation} \label{EqFrameD}
|\bra{f,\pi(D_{\gamma}z)g}|=|\bra{f,M_{a/\gamma}T_{b\gamma}g}|=\gamma|\bra{f_{\gamma},M_{a}T_{b}g_{\gamma}}|=\gamma|\bra{f_{\gamma},\pi(z)g_{\gamma}}|.
\end{equation}

Let $\lambda\in\mathbb R\setminus\{0\}$ be a parameter, we use the notation
\[h(x;\lambda)=e^{-\pi i\lambda x^2}\]
for the chirp with parameter $\lambda$.

\section{Metaplectic operators and extended Gaussians} \label{SecMain}
A matrix $A\in\mathbb R^{2\times 2}$ is called a \emph{symplect matrix} if it keeps the symplectic form, i.e., $[Az,Az']=[z,z']$ holds for all $z,z'\in\mathbb R^2$. If $S=\begin{pmatrix}0 & 1 \\ -1 & 0\end{pmatrix}$, then since $[z,z']=\bra{z,Sz'}$, $A$ being a symplectic matrix is equivalent to $A^TSA=S$,  which is further equivalent to $\det(A)=1$. This can be seen from the straightforward computation
\[\begin{pmatrix}a & c \\ b & d\end{pmatrix}\begin{pmatrix}0 & 1 \\ -1 & 0\end{pmatrix}\begin{pmatrix}a & b \\ c & d\end{pmatrix}=\begin{pmatrix}0 & ad-bc \\ -(ad-bc) & 0\end{pmatrix}.\]
A \emph{metaplectic operator} $\mu(A)$ on $L^2(\mathbb R)$ associated with a symplectic matrix $A$ is a unitary linear operator that satisfies
\begin{equation} \label{EqDefMu}
\rho(Az)\mu(A)=\mu(A)\rho(z),\quad \forall z\in \mathbb R^2.
\end{equation}
If $A=S=\begin{pmatrix}0 & 1 \\ -1 & 0\end{pmatrix}$, then from \eqref{EqPiRho} and the relation $\mathcal FM_aT_b=T_aM_{-b}$ one can derive that $\mu(A)=\mathcal F$. Alternatively if $A=D_{\gamma}=\diag(1/\gamma,\gamma)$ is diagonal, then a change of variable indicates that $\mu(A)$ is a scaled (to be unitary) dilation.

In general, it is not simple to show that every symplectic matrix $A$ has an associated metaplectic operator (see \cite[Chapter 4]{folland1989} or \cite[Chapter 9]{groechenig2001}). For the $L^2(\mathbb R)$ case there is an elementary way to realize these operators explicitly: 

Let $f\mapsto h(\cdot\;;\lambda')*f(\cdot)$ be the multiplication by $\hat h_{\lambda'}$ in the frequency domain:
\[\left(h(\cdot\;;\lambda')*f(\cdot)\right)(x)=\mathcal F^{-1}\left(\widehat{h(\cdot\;;\lambda')}\cdot \widehat{f(\cdot)}\right)(x), \quad f\in L^2(\mathbb R).\]
Recall that (see the table in \cite[A-7]{kammler2000})
\begin{equation} \label{EqFH}
\widehat{h(\cdot\;;\lambda')}=c_{\lambda'}\cdot h(\cdot\;;-\frac{1}{\lambda'}),
\end{equation}
where $c_{\lambda'}=|\lambda'|^{-1/2}e^{-\sgn(\lambda')\pi i/4}=|\lambda'|^{-1/2}e^{\pm\pi i/4}$ comes from taking a square root of $\pm i$. Set
\[U_{\lambda}=\begin{pmatrix}1 & \lambda \\ 0 & 1 \end{pmatrix}, \quad L_{\lambda'}=\begin{pmatrix}1 & 0 \\ \frac{1}{\lambda'} & 1 \end{pmatrix}.\]
Then for any $z=(a,b)^T\in\mathbb R^2$ and any $f\in L^2(\mathbb R)$ we can readily compute that
\begin{align}
\pi(z)(h(x;\lambda)\cdot f(x))&=e^{2\pi ia x}\cdot e^{-\pi i\lambda(x-b)^2}\cdot f(x-b) \\
&=e^{2\pi ia x}\cdot e^{-\pi i\lambda x^2}\cdot e^{2\pi i\lambda b x}\cdot e^{-\pi i\lambda b^2} \cdot f(x-b) \\
&=e^{-\pi i\lambda b^2}\cdot h(x;\lambda) \cdot (M_{a+\lambda b}T_bf)(x) \\
&=e^{-\pi i\lambda b^2}\cdot h(x;\lambda) \cdot\left(\pi(U_{\lambda}z)f\right)(x), \label{EqCPMT}
\end{align}
and
\begin{align*}
\pi(z)(h(\cdot\;;\lambda')*f(\cdot))&=\mathcal F^{-1}\mathcal F\left(\pi(z)(h(\cdot\;;\lambda')*f(\cdot))\right) & \\
&=\mathcal F^{-1}\left(T_aM_{-b}(c_{\lambda'}\cdot h(\cdot\;;-\frac{1}{\lambda'})\cdot \hat f(\cdot))\right) & \text{by \eqref{EqFH}} \\
&=c_{\lambda'}\cdot e^{2\pi iab}\cdot \mathcal F^{-1}\left(M_{-b}T_a\left(h(\cdot\;;-\frac{1}{\lambda'})\cdot \hat f(\cdot)\right)\right) & \text{by \eqref{EqComm}} \\
&=c_{\lambda'}\cdot e^{2\pi iab}\cdot e^{\frac{\pi ia^2}{\lambda'}}\cdot \mathcal F^{-1}\left(h(\cdot\;;-\frac{1}{\lambda'})\cdot (M_{-b-\frac{a}{\lambda'}}T_a\hat f)(\cdot)\right) & \text{by \eqref{EqCPMT}} \\
&=e^{2\pi iab}\cdot e^{\frac{\pi ia^2}{\lambda'}}\cdot h(\cdot\;;\lambda')*(T_{b+\frac{a}{\lambda'}}M_af)(\cdot) & \text{by \eqref{EqFH}} \\
&=e^{-\frac{\pi ia^2}{\lambda'}}\cdot h(\cdot\;;\lambda')*(M_aT_{\frac{a}{\lambda'}+b}f)(\cdot) & \text{by \eqref{EqComm}} & \\
&=e^{-\frac{\pi ia^2}{\lambda'}}\cdot h(\cdot\;;\lambda')*\left(\pi(L_{\lambda'}z)f\right)(\cdot). &
\end{align*}
By applying an LU decomposition on $A$, followed by appropriate row/column scalings on the lower and upper triangular parts respectively, any symplectic matrix $A$ can be written into $A=LDU$ for some lower triangular matrix $L$, some upper triangular matrix $U$, and some diagonal matrix $D$ such that the main diagonal of both $L$ and $U$ are $1$, and $\det(D)=1$. Therefore with proper scalings $\mu(A)=\mu(L)\mu(D)\mu(U)$ can be accordingly constructed by composing the multiplications of chirps in the time domain and in the frequency domain respectively, with a dilation in the middle. This type of decomposition remains true in higher dimensions, see the integral formula in \cite[Theorem 4.51]{folland1989}.

\begin{Prop}(\cite[Proposition 4.73]{folland1989}) \label{PropEG}
Let $f\in L^2(\mathbb R)$, then $f$ is an extended Gaussian $\varphi(\cdot\;;\boldsymbol a,\boldsymbol b,\boldsymbol c)$ if and only if there are some $\boldsymbol C\in\mathbb C$, some symplectic matrix $A$, and some $p,q\in\mathbb R$ such that
\[f=\boldsymbol C\rho(p,q)\mu(A)\varphi\] 
holds. In particular, $\boldsymbol b=0$ if and only if $p=q=0$. 
\end{Prop}

The relation between $A,C$ and $\boldsymbol a, \boldsymbol b, \boldsymbol c$ in Proposition \ref{PropEG} can be determined explicitly by computing out $h(\cdot\;;\lambda')*\varphi(\cdot)$: Let
\[u=\frac{\gamma^2}{\gamma^4+(\frac{1}{\lambda'})^2}, \quad v=\frac{\frac{1}{\lambda'}}{\gamma^4+(\frac{1}{\lambda'})^2}.\]
Recall the formula (see \cite[A-7]{kammler2000}) that if $\alpha,\beta\in\mathbb R$, $\alpha+\beta i\neq 0$ and $\alpha\ge |\beta|$, then
\begin{equation} \label{EqFC}
\mathcal F(e^{-\pi(\alpha+\beta i)^2x^2})(\xi)=\frac{1}{\alpha+\beta i}e^{-\frac{\pi\xi^2}{(\alpha+\beta i)^2}}.
\end{equation}
Observe that
\[h(\xi;-\frac{1}{\lambda'})\cdot\varphi_{\gamma}(\xi)=e^{\frac{\pi i\xi^2}{\lambda'}}\cdot e^{-\pi\gamma^2\xi^2}=e^{-\pi \xi^2(\gamma^2-\frac{1}{\lambda'}i)}=e^{\frac{-\pi \xi^2\left(\gamma^4+(\frac{1}{\lambda'})^2\right)}{\gamma^2+\frac{1}{\lambda'}i}}=e^{-\frac{\pi\xi^2}{u+vi}}.\]
Clearly $u>0$, which means that $u+vi$ is in the right half complex plane. Hence there exists some $\eta=\alpha+\beta i\neq 0$ with $\alpha\ge |\beta|$ (i.e., the complex angle of $\eta$ is between $-\pi/4$ and $\pi/4$) so that $\eta^2=u+vi$. It then follows from \eqref{EqFH} and \eqref{EqFC} that
\begin{align*}
\left(h(\cdot\;;\lambda')*\varphi_{\frac{1}{\gamma}}(\cdot)\right)(x)&=\mathcal F^{-1}\left(\widehat{h(\cdot\;;\lambda')}\cdot\widehat{\varphi_{\frac{1}{\gamma}}(\cdot)}\right)(x) \\
&=\mathcal F^{-1}\left(c_{\lambda'}h(\cdot\;;-\frac{1}{\lambda'})\cdot\gamma \varphi_{\gamma}(\cdot)\right)(x) \\
&=c_{\lambda'}\gamma\eta\cdot e^{-\pi(u+vi)x^2} \\
&=c_{\lambda'}\gamma\eta\cdot h(x;v)\cdot \varphi_{\sqrt u}(x).
\end{align*}

\section{Zak transforms of extended Gaussians}
The \emph{Zak transform} of $f\in L^2(\mathbb R)$ is a function on $\mathbb R^2$ defined as
\[Zf(t,\omega)=\sum_{k\in\mathbb Z}f(t-k)e^{2\pi ik\omega}.\]
It follows immediately from the definition that $Zf$ is continuous for $f\in W_0(\mathbb R)$. The unit square $[0,1)^2$ is called the \emph{fundamental domain} of $Z$ since not only that $Z$ is an isometry from $L^2(\mathbb R)$ to $L^2\left([0,1)^2\right)$ (it maps the orthonormal Gabor basis element $M_jT_k\chi_{[0,1)}$ to the orthonormal Fourier basis element $e^{2\pi i(jt+k\omega)}$) but also it has the quasi-periodicity that
\[Zf(t-j,\omega)=e^{-2\pi i j\omega}Z(t,\omega), \quad Zf(t,\omega-k)=Z(t,\omega),  \quad \forall j,k\in\mathbb Z,\]
which indicates that the values of $Zf$ on $\mathbb R^2$ are completely determined by its values on the unit square $[0,1)^2$. For $(a,b)\in\mathbb R^2$, the Zak transform intertwines with time-frequency shifts in the following way:
\begin{equation} \label{EqZMT}
Z(M_aT_bf)(t,\omega)=e^{2\pi ia(x-b)}Zf(t-b,\omega-a).
\end{equation}
See \cite[Chapter 8]{groechenig2001} for more of its properties.

We shall call $(t_0,\omega_0)$ a \emph{simple zero} of $Zf$ if $Zf(t_0,\omega_0)=0$ and there is a neighbourhood $\Omega$ of $(t_0,\omega_0)$ and constants $A,B$ such that $Zf$ is continuous on $\Omega$ and
\begin{equation} \label{EqSZ}
0<A\le \frac{|Zf(t,\omega)|}{\sqrt{(t-t_0)^2+(\omega-\omega_0)^2}}\le B<\infty
\end{equation}
holds for all $(t,\omega)\in\Omega\setminus\{(t_0,\omega_0)\}$. 

Theta functions are elliptic analogs of exponential functions. For convenience we will use the following form of the theta function, which is consistent with the convention and results in \cite{baxter1996}:
\[\Theta(\boldsymbol z,\boldsymbol q)=\sum_{k\in\mathbb Z}\boldsymbol q^{k^2}\boldsymbol z^k, \quad \boldsymbol z\in\mathbb C\setminus\{0\}, \boldsymbol q\in\mathbb C, |\boldsymbol q|<1.\]
This form is closest to the usual definition of
\[\Theta_3(\boldsymbol z,\boldsymbol q)=\sum_{k\in\mathbb Z}\boldsymbol q^{k^2}e^{2ki\boldsymbol z},\]
and can be easily translated into each other. As in \cite{baxter1996}, we also imposed the extra condition $|\boldsymbol q|<1$. Otherwise the propositions below will hinge on absolute convergence of $\sum_k|\boldsymbol q|^k$ and $\prod_k|1-\boldsymbol q^k|$.

This form of the theta function is closely related to the Zak transform of standard Gaussians (see e.g.,\cite{janssen2002}). In this article the following decomposition is relevant (\cite[Lemma 2.1]{baxter1996}, also available in \cite[Chapter 21.3, p.469]{whittaker1902}):
\[\Theta(\boldsymbol z,\boldsymbol q)=\left(\prod_{j=1}^{\infty}(1-\boldsymbol q^{2j})\right)\cdot\left(\prod_{k=0}^{\infty}(1+\boldsymbol q^{2k+1}z)(1+\boldsymbol q^{2k+1}\boldsymbol z^{-1})\right).\]
From this decomposition we immediately obtain that 
\begin{Prop}(\cite[Corollary 2.2]{baxter1996}) \label{PropTheta}
$\Theta(\boldsymbol z,\boldsymbol q)=0$ is only attained at $\boldsymbol z=-\boldsymbol q^m$ where $m$ can be any odd integer.
\end{Prop}

If we fix $\boldsymbol q$, take $\boldsymbol z=e^{2i\boldsymbol x}$ for $\boldsymbol x\in\mathbb C$ and set
\[H(\boldsymbol x)=\Theta(e^{2i\boldsymbol x},\boldsymbol q),\]
then we have 
\begin{Prop}(\cite[Chapter 21.3, p.469]{whittaker1902}) \label{PropTheta2}
$H(\boldsymbol x)$ is an entire function on $\mathbb C$ and all zeros of $H(\boldsymbol x)$ are simple zeros.
\end{Prop}

\begin{Lm} \label{LmZero}
Let $\lambda\in\mathbb R\setminus\{0\}$ and $\gamma>0$, then $Z\left(h(\cdot\;;\lambda)\cdot\varphi_{\gamma}(\cdot)\right)$ has a unique simple zero at the center $(1/2,1/2)$ of the unit square $[0,1)^2$.
\end{Lm}

\begin{proof}
We readily compute
\begin{align*}
Z\left(h(\cdot\;;\lambda)\cdot\varphi_{\gamma}(\cdot)\right)(t,\omega)&=\sum_{k\in\mathbb Z}e^{-\pi i\lambda(t-k)^2}\cdot e^{-\pi\gamma^2(t-k)^2}\cdot e^{2\pi ik\omega} \\
&=\sum_{k\in\mathbb Z}e^{-\pi i\lambda t^2}\cdot e^{2\pi i\lambda tk}\cdot e^{-\pi i\lambda k^2}\cdot e^{-\pi\gamma^2 t^2}\cdot e^{2\pi\gamma^2 tk}\cdot e^{-\pi\gamma^2 k^2}\cdot e^{2\pi ik\omega} \\
&=e^{-\pi(\gamma^2+\lambda i)t^2}\sum_{k\in\mathbb Z}e^{-\pi(\gamma^2+\lambda i)k^2}\cdot e^{2\pi\left(\gamma^2t+(\omega+\lambda t)i\right)k} \\
&=e^{-\pi(\gamma^2+\lambda i)t^2}\Theta(\boldsymbol z,\boldsymbol q),
\end{align*}
where
\[\boldsymbol z=e^{2\pi\left(\gamma^2t+(\omega+\lambda t)i\right)}, \quad \boldsymbol q=e^{-\pi(\gamma^2+\lambda i)}.\]
This is well defined since $|\boldsymbol q|<1$ and $\boldsymbol z\neq 0$ are obvious. By Proposition \ref{PropTheta}, zeros are only attained at $\boldsymbol z=-\boldsymbol q^m$ for odd integers $m$. Taking absolute values at both sides we get
\[e^{2\pi \gamma^2t}=|\boldsymbol z|=|\boldsymbol q^m|=e^{-\pi \gamma^2m}.\]
For $t\in[0,1)$ the equation can only hold at $t=1/2$ ($m$ will not be an odd integer for other values of $t$), in which case $m=-1$. Plugging this back in we obtain
\[e^{2\pi\left(\frac{\gamma^2}{2}+(\omega+\frac{\lambda}{2})i\right)}=\boldsymbol z=-\boldsymbol q^{-1}=-e^{\pi(\gamma^2+\lambda i)},\]
i.e., $e^{2\pi\omega i}=-1$, and for $\omega\in[0,1)$ this is only attained at $\omega=1/2$. Therefore $(1/2,1/2)$ is the only zero of $Z(h_{\lambda}\cdot\varphi_{\gamma})$ in the unit square $[0,1)^2$. \footnote{Since $h(x;\lambda)\cdot\varphi_{\gamma}(x)$ is an even function, the existence of a zero at the center of the unit square $[0,1)^2$ can also be asserted by Proposition \ref{PropZero}\ref{Prop0E} in the appendix, but the uniqueness of this zero is non-trivial.}

Finally since $\boldsymbol q$ does not depend on the variables $t,\omega$, by Proposition \ref{PropTheta2} we see that the zero is simple with respect to the complex variable
\[\boldsymbol x=\frac{2\pi\left(\gamma^2 t+(\omega+\lambda t)i\right)}{2i}=\pi(\omega+\lambda t-\gamma^2 ti).\]
We set
\[z=\begin{pmatrix}\omega \\ t \end{pmatrix}, \quad z_0=\begin{pmatrix}1/2 \\ 1/2 \end{pmatrix}, \quad U=\pi\cdot\begin{pmatrix}1 & \lambda \\ 0 & -\gamma^2\end{pmatrix}, \quad z_1=U^{-1}z_0,\]
where $U^{-1}$ exists since $\gamma\neq0$. Then there are constants $C,C'$ so that on a neighborhood $\Omega$ of $z_0$ we have ($\|\cdot\|$ is the Euclidean norm)
\begin{align*}
|Z\left(h(\cdot\;;\lambda)\cdot\varphi_{\gamma}(\cdot)\right)(t,\omega)|^2\ge C|x-\frac{1}{2}-\frac{1}{2}i|^2&=C(\pi\omega+\pi\lambda t-\frac{1}{2})^2+C(\pi\gamma^2t+\frac{1}{2})^2 \\
&=C\|Uz-z_0\|^2=C\|U(z-z_1)\|^2\ge C'\|z-z_1\|^2,
\end{align*}
where the first inequality holds since by Proposition \ref{PropTheta2} the zero is simple with respect to the variable $x$, and the last inequality holds since $\gamma\neq0$ and thus $U:\mathbb R^2\to\mathbb R^2$ is lower bounded. 

Now if $z_1=z_0$, then clearly $\|z-z_1\|^2=\|z-z_0\|^2$ holds for any $z$. Otherwise there exists a subneighborhood $\Omega'$ of $z_0$ that does not contain $z_1$. Then for any $z\in\Omega'$, $\|z-z_1\|^2$ is lower bounded while $\|z-z_0\|^2$ is upper bounded. Hence there exists $C''$ so that
\[\|z-z_1\|^2\ge C''\|z-z_0\|^2, \quad \forall z\in\Omega.\]
This builds a lower bound for the simplicity of the zero required in \eqref{EqSZ}. 

Next, both partial derivatives
\[\frac{\partial}{\partial t}Z\left(h(\cdot\;;\lambda)\cdot\varphi_{\gamma}(\cdot)\right)(t,\omega)=\frac{\partial}{\partial t}\left(\sum_{k\in\mathbb Z}e^{-\pi i\lambda(t-k)^2}\cdot e^{-\pi\gamma^2(t-k)^2}\cdot e^{2\pi ik\omega}\right)=-2\pi(\gamma^2+\lambda i)\cdot\sum_{k\in\mathbb Z}(t-k)\cdot e^{-\pi(\gamma^2+\lambda i)(t-k)^2}\cdot e^{2\pi ik\omega}\]
and 
\[\frac{\partial}{\partial \omega}Z\left(h(\cdot\;;\lambda)\cdot\varphi_{\gamma}(\cdot)\right)(t,\omega)=\frac{\partial}{\partial \omega}\left(\sum_{k\in\mathbb Z}e^{-\pi i\lambda(t-k)^2}\cdot e^{-\pi\gamma^2(t-k)^2}\cdot e^{2\pi ik\omega}\right)=2\pi i k\cdot\sum_{k\in\mathbb Z}e^{-\pi(\gamma^2+\lambda i)(t-k)^2}\cdot e^{2\pi ik\omega}\]
exist and are continuous (since $h(x;\lambda)\cdot\varphi_{\gamma}(x)$ has exponential decay), therefore an upper bound required in \eqref{EqSZ} can be obtained from differentiability. Together we can conclude that $(1/2,1/2)$ is indeed a simple zero of $Z\left(h(\cdot\;;\lambda)\cdot\varphi_{\gamma}(\cdot)\right)$.
\end{proof}

\section{Main Result}
Let $\Lambda\subset\mathbb R^2$ be a \emph{uniformly discrete set}, i.e., 
\[\inf_{\substack{\lambda\neq \lambda' \\ \lambda,\lambda'\in\Lambda}}\|\lambda-\lambda'\|>0.\]
Fix a compact set $E\subset\mathbb R^2$ of measure $1$, let
\[n^-(r)=\min_{t\in\mathbb R^2}|T_t(rE)\cap \Lambda|, \quad r>0\]
be the smallest number of points in the intersection of $\Lambda$ among all possible translations of $rE$, the \emph{lower uniform density} \cite{landau1967} of $\Lambda$ is then defined as
\[D^-(\Lambda)=\liminf_{r\to\infty}\frac{n^-(r)}{r^2}.\]
If $A\in\mathbb R^{2\times 2}$ is non-singular, and $\Lambda=A\mathbb Z^2$ is a lattice, then $D^-(\Lambda)=1/\det(A)$.

The following is essentially \cite[Theorem 1.1]{seip1992, wallsten1992}, see also the exposition in \cite[Theorem 14, A.1, A.2]{heil2007} for the relation between the Gabor frame set of the standard Gaussian and sampling sets in Bargmann-Fock spaces.
\begin{Prop}(\cite[Theorem 1.1]{seip1992, wallsten1992}) \label{PropMFG}
Let $\Lambda$ be a uniformly discrete set, then $\gframe{\varphi}{\Lambda}$ is a frame if and only if $D^-(\Lambda)>1$.
\end{Prop} 

\begin{Thm} \label{ThmMain}
Let $\boldsymbol a,\boldsymbol b,\boldsymbol c\in\mathbb C$ with $\re(\boldsymbol a)<0$,  then the extended Gaussian $\varphi(x;\boldsymbol a,\boldsymbol b,\boldsymbol c)=e^{\boldsymbol ax^2+\boldsymbol bx+\boldsymbol c}$ has maximal frame set, and its Zak transform has a unique simple zero in the unit square $[0,1)^2$. In particular, if $\boldsymbol b=0$, then the zero is at $(1/2,1/2)$.
\end{Thm}

\begin{proof}
Maximality of the frame set: Let $\alpha,\beta$ be an arbitrary pair that satisfies $\alpha,\beta>0$ and $\alpha\beta<1$. By Proposition \ref{PropEG}, there is some $C\in\mathbb C$, some $p,q\in\mathbb R$ and some symplectic matrix $A$ such that $\varphi(\cdot\;;\boldsymbol a,\boldsymbol b,\boldsymbol c)=C\rho(p,q)\mu(A)\varphi(\cdot)$ holds. Since scalings and time-frequency shifts do not change the frame set, we see that $\glframe{\varphi(\cdot\;;\boldsymbol a,\boldsymbol b,\boldsymbol c)}{\alpha}{\beta}$ is a frame if and only if $\glframe{\mu(A)\varphi}{\alpha}{\beta}$ is a frame. Now if $x=(x_1,x_2)\in\diag(\alpha,\beta)\mathbb Z^2, \tilde x=(x_2,x_1)^T$, $\tilde y=(y_2,y_1)^T=A^{-1}\tilde x$ and $y=(y_1,y_2)^T$, then combining \eqref{EqPiRho} and \eqref{EqDefMu} we get
\[\pi(x)\mu(A)\varphi=e^{\pi ix_1x_2}\rho(\tilde x)\mu(A)\varphi=e^{\pi ix_1x_2}\rho(A\tilde y)\mu(A)\varphi=e^{\pi ix_1x_2}\mu(A)\rho(\tilde y)\varphi=e^{\pi ix_1x_2}\mu(A)e^{-\pi iy_1y_2}\pi(y)\varphi,\]
which further leads to
\[\bra{f,\pi(x)\mu(A)\varphi}=\bra{f,e^{\pi ix_1x_2}\mu(A)e^{-\pi iy_1y_2}\pi(y)\varphi}=\bra{e^{-\pi ix_1x_2}\mu(A^T)e^{\pi iy_1y_2}f,\pi(y)\varphi}.\]
The map $f\mapsto e^{-\pi ix_1x_2}\mu(A^T)e^{\pi iy_1y_2}f$ is unitary and does not change the $L^2$ norm, therefore $\glframe{\mu(A)\varphi}{\alpha}{\beta}$ being a frame is further equivalent to $\gframe{\varphi}{A^{-1}\diag(\alpha,\beta)\mathbb Z^2}$ being a frame. The lower uniform density of $A^{-1}\diag(\alpha,\beta)\mathbb Z^2$ is $1/(\det(A)\alpha\beta)$ which is strictly larger than $1$ since $\det(A)=1$ and $\alpha\beta<1$, therefore by Proposition \ref{PropMFG}, $\gframe{\varphi}{A^{-1}\diag(\alpha,\beta)\mathbb Z^2}$ is indeed a frame.

Zeros of the Zak transform: If $\boldsymbol a\in\mathbb C$ has negative real part, then we can write it into $\boldsymbol a=-\pi\gamma^2-\lambda\pi i$ for some $\gamma>0$ and some $\lambda\in\mathbb R\setminus\{0\}$. Therefore Lemma \ref{LmZero} shows that $Ze^{\boldsymbol ax^2}=Z\left(h(\cdot\;;\lambda)\cdot\varphi_{\gamma}(\cdot)\right)$ has a unique simple zero at the center of the unit square $[0,1)^2$. And the equality $Ze^{\boldsymbol ax^2+\boldsymbol c}=e^{\boldsymbol c}Ze^{\boldsymbol ax^2}$ further indicates that the same is true for $Ze^{\boldsymbol ax^2+\boldsymbol c}$, this proves the $\boldsymbol b=0$ case. If $\boldsymbol b\neq0$, then by Proposition \ref{PropEG}, there is some $\boldsymbol C\in\mathbb C$, some $p,q\in\mathbb R$ such that $e^{\boldsymbol ax^2+\boldsymbol bx+\boldsymbol c}=\boldsymbol C\rho(p,q)e^{\boldsymbol ax^2+\boldsymbol c}$, then \eqref{EqZMT} shows zeros of $Ze^{\boldsymbol ax^2+\boldsymbol bx+\boldsymbol c}$ differ from zeros of $Ze^{\boldsymbol ax^2+\boldsymbol c}$ only by a translation of $(p,q)$, thus it still has a unique simple zero at $(1/2-p,1/2-q) \bmod \mathbb Z^2$ in the unit square $[0,1)^2$.
\end{proof}

It is also clear that $\varphi(\cdot\;;\boldsymbol a,\boldsymbol b,\boldsymbol c)\in W_0(\mathbb R)$, therefore by the Amalgam Balian-Low theorem, the critical case $\alpha\beta=1$ is not in the frame set of $\varphi(\cdot\;;\boldsymbol a,\boldsymbol b,\boldsymbol c)$.

\section*{Appendix: Facts about zeros of the Zak transform}
It easily follows from \eqref{EqFrameD} that $\glframe{g}{\alpha}{\beta}$ is a frame if and only if $\glframe{g_{1/\beta}}{\alpha\beta}{}$ is. The Zak transform method is then particularly effective in analyzing frame sets of the critical case and other integer oversampling cases since in such cases it diagonalizes the frame operator and makes it a multiplication operator \cite[Corollary 8.3.2, Equation 8.21]{groechenig2001}, i.e., if $\alpha=1/n$ with $n\in\mathbb N$, then $\glframe{g_{\gamma}}{\alpha}{}$ is a frame if and only if there are constants $A,B$ such that $0<A\le \sum_{j=0}^{n-1}|Zg_{\gamma}(t-\frac{j}{n}, \omega)|^2\le B<\infty$ holds for almost every $t,\omega$. Therefore in these cases locations of zeros of $Zg_{\gamma}$ play a prominent role, since as long as $Zg_{\gamma}$ is continuous, then its essential lower boundedness on the unit square $[0,1)^2$ (which by quasi-periodicity is equivalent to its global essential lower boundedness) is equivalent to having no zero in it. The other cases are more complicated. But in general they can be characterized as (\cite{janssen1995}, also stated as \cite[Theorem 2.3]{groechenig2018}): $\glframe{g_{\gamma}}{\alpha}{}$ is a frame if and only if there are constants $A,B>0$ such that
\begin{equation} \label{EqCharR}
A\|f\|_{L^2([0,1))}^2\le\sum_{k\in\mathbb Z}\left|\int_0^1f(\omega)Zg_{\gamma}(t-k\alpha,
\omega)d\omega\right|^2\le B\|f\|_{L^2([0,1))}^2
\end{equation}
holds for all $f\in L^2([0,1))$ and all $t\in[0,1)$.

To understand how the estimates change upon dilation is often a major difficulty since it is hard to relate and control $Zg_{\gamma}$ through $Zg$. Another issue is that unlike in the critical case or in other integer oversampling cases, here merely lifting the lower bound of $|Zg_{\gamma}|$ does not seem to help. Indeed, for an individual $g_{\gamma}$, if we fix $t$, then for each $k\in\mathbb Z$, the function $h_k^{(t)}(\omega)=Zg_{\gamma}(t-k\alpha,\omega)$ is a univariate function of $\omega$, and the sum in \eqref{EqCharR} becomes $\sum_k|\bra{f,\overline{h_k^{(t)}}}_{L^2([0,1))}|^2$. Therefore \eqref{EqCharR} is equivalent to requiring $\{h_k^{(t)}\}_k$ to be a frame of $L^2([0,1))$, but even if each $|h_k^{(t)}|$ is away from $0$, together they can still be collided together for some $f$ to be orthogonal to all of them to make the sum in \eqref{EqCharR} vanish.

Despite that, as mentioned, mysteriously so far all functions that found to have maximal frame sets possess at most one simple zero in the unit square $[0,1)^2$. Hence we would like to provide an overview of what is known (other than Theorem \ref{ThmMain}) about existence, simplicity and uniqueness of zeros of Zak transforms.
\begin{Prop} \label{PropZero}
Existence:
\begin{enumerate}[leftmargin=*, label=(\alph*)]
\item If $Zf$ is continuous, then it has a zero in the unit square $[0,1)^2$. \label{Prop0C}
\item If $f$ is real valued, then zeros of $Zf$ in the unit square $[0,1)^2$ are symmetric with respect to the line $\omega=1/2$. If in addition $Zf$ is continuous, then $Zf$ must have a zero on the line $\omega=1/2$. \label{Prop0R}
\item If $f$ is an even function, then $Zf(1/2,1/2)=0$. \label{Prop0E}
\item If $f$ is an odd function, then $Zf(0,0)=Zf(1/2,0)=Zf(0,1/2)=0$. \label{Prop0O}
\item If $f,\hat f\in W(\mathbb R)$ and $f$ is an eigenfunction of $\mathcal F$ whose eigenvalue is not $-i$, then $Zf(1/2,1/2)=0$. \label{Prop0F}
\end{enumerate}

Uniqueness and simplicity:
\begin{enumerate}[leftmargin=*, label=(\roman*)]
\item If $f$ is even, continuous, and can be written as $g+T_1g$ on $[0,+\infty)$ for some strictly convex, non-negative $g\in L^1(\mathbb R)$, then $Zf$ has a unique zero in the unit square $[0,1)^2$. \label{Prop0U}
\item If $f\in L^1(\mathbb R)$ is non-negative, supported on and strictly decreasing on $[0,+\infty)$, then there is some constant $A$ so that $|Zf|>A>0$ holds everywhere. \label{Prop0N}
\item If $f\in L^1(\mathbb R)$ is even, continuous, non-negative and strictly convex on $[0,+\infty)$, then $Zf$ has no zero on $[0,1)^2\setminus\{(t,\omega):t=1/2\}$. \label{Prop0U2}
\item If $f$ is an exponential $B$-spline of finite order or a totally positive function (excluding the one sided exponential and its translations) or the convolution of these two types of functions, then $Zf$ has a unique zero in the unit square $[0,1)^2$, the zero is simple and is on the line $\omega=1/2$. \label{Prop0S}
\end{enumerate}
\end{Prop}

\ref{Prop0C} is a classical result, and we refer to \cite[Lemma 8.4.2]{groechenig2001} for a proof and references therein for its origin. \ref{Prop0R} to \ref{Prop0F} are mentioned at many places. But they are often mixed with stronger assumptions. We give a simple sketch here: 

For \ref{Prop0R}, if $f$ is real valued then
\[\overline{Zf(t,\omega)}=\sum_{k\in\mathbb z}\overline{f(t-k)e^{2\pi ik\omega}}=\sum_{k\in\mathbb z}f(t-k)e^{-2\pi ik\omega}=Zf(t,-\omega)=Zf(t,1-\omega),\]
where the last equality follows from the quasi-periodicity. This indicates that all zeros are symmetric with respect to the line $\omega=1/2$. If $Zf$ is continuous, then on this line we have that
\[Zf(t,\frac{1}{2})=\sum_{k\in\mathbb Z}f(t-k)\cdot (-1)^k,\]
is a real valued continuous function of $t$ that satisfies $Zf(0,1/2)=-Zf(1,1/2)$, which leads to the existence of a zero for some $t\in[0,1)$ by the intermediate value property.

For \ref{Prop0E}, if $f$ is even, then
\[Zf(t,\omega)=\sum_{k\in\mathbb Z}f(t-k)e^{2\pi ik\omega}=\sum_{k\in\mathbb Z}f(k-t)e^{2\pi ik\omega}.\]
Replacing $k$ with $-j$ we obtain
\[\sum_{k\in\mathbb Z}f(k-t)e^{2\pi ik\omega}=\sum_{j\in\mathbb Z}f(-j-t)e^{-2\pi ij\omega}=\sum_{j\in\mathbb Z}f(j+t)e^{-2\pi ij\omega}.\]
Plugging in $(t,\omega)=(1/2,1/2)$ we see that both equations are signed sums of $f$ at half integers. But they have opposite signs, i.e., $Zf(1/2,1/2)=-Zf(1/2,1/2)$, which means it is $0$.

For \ref{Prop0O}, if $f$ is odd, then similar as above we have
\[Zf(t,\omega)=-\sum_{k\in\mathbb Z}f(k-t)e^{2\pi ik\omega}=\sum_{j\in\mathbb Z}f(j+t)e^{-2\pi ij\omega}.\]
Plugging in $(t,\omega)=(1/2,0)$ we get sums of $f$ at half integers with opposite signs. Thus $Zf(1/2,0)=0$, and it is also clear from oddity by direct computation that $Zf(0,0)=Zf(0,1/2)=f(0)=0$.

For \ref{Prop0F}, suppose that $\lambda\neq-i$ is the eigenvalue. Then we simply apply the formula (\cite[Proposition 8.2.2]{groechenig2001}) that
\[Zf(t,\omega)=e^{2\pi it\omega}Z\hat f(\omega,-t), \quad f,\hat f\in W(\mathbb R)\]
followed by the quasi-periodicity to get
\[Zf(t,\omega)=e^{2\pi it\omega}Z\hat f(\omega,-t)=\lambda e^{2\pi it\omega}Zf(\omega,-t)=\lambda e^{2\pi it\omega}Zf(\omega,1-t).\]
Plugging in $t=\omega=1/2$ we obtain $Zf(1/2,1/2)=i\lambda Zf(1/2,1/2)$, and $i\lambda \neq 1$ since $\lambda\neq -i$, which shows \ref{Prop0F}.

\ref{Prop0U} to \ref{Prop0U2} are in \cite[Section 3.2.1-3.2.3]{janssen2003a}. \ref{Prop0S} is established in \cite{baxter1996} for Gaussians, in \cite{kloos2014} for exponential $B$-splines of finite order and for finite type totally positive functions without the Gaussian factor, in \cite{kloos2015} for infinite type totally positive functions without the Gaussian factor, and finally in \cite{vinogradov2017} for totally positive functions with the Gaussian factor and for their convolutions with exponential $B$-splines.

\end{document}